\documentclass[11pt,letterpaper]{article}
\usepackage{amsthm}
\usepackage{amsmath}
\usepackage{amsfonts}
\usepackage{latexsym}
\usepackage{geometry}
\usepackage{enumerate}
\usepackage{csquotes}

\emergencystretch=\maxdimen
\title{On three-dimensional Type I $\kappa$-solutions to the Ricci flow}
\author{Yongjia Zhang}
\begin{document}
\maketitle

In this short note, we prove that the only simply connected noncompact three-dimensional Type I $\kappa$-solution to the Ricci flow is the shrinking cylinder. This work can be regarded as a generalization of Cao, Chow, and Zhang \cite{cao2016three}, and a complement of Ding \cite{ding2009remark} and Ni \cite{ni2009closed}. Up to this point, three-dimensional $\kappa$-solutions of Type I are completely classified, and it remains interesting to work further towards Perelman's assertion, that the only remaining possibility of three-dimensional noncompact $\kappa$-solution is the Bryant soliton; see \cite{perelman2002entropy}. Brendle \cite{brendle2013rotational} is working to that end. The classification of three-dimensional $\kappa$-solution is of importance to the study of four-dimensional Ricci flows, because of a possible dimension-reduction procedure.
\\

We remind the reader of the following definition.

\newtheorem{Definition_1}{Definition}
\begin{Definition_1}\label{Definition_1}
An ancient solution to the Ricci flow $(M,g(t))_{t\in(-\infty,0]}$ is called a $\kappa$-solution if it is $\kappa$-noncollapsed on all scales and has bounded curvature on every time slice. A $\kappa$-solution is called Type I if its Riemann curvature tensor satisfies
\begin{eqnarray} \label{Type I}
|Rm|(g(t))\leq\frac{C}{|t|},
\end{eqnarray}
for all $t\in(-\infty,0)$, where $C$ is a constant that does not depend on $t$.
\end{Definition_1}

It is well-known that every three-dimensional $\kappa$-solution has uniformly bounded and nonnegative sectional curvature.
\\

Our main theorem is the following.

\newtheorem{Main_Theorem}[Definition_1]{Theorem}
\begin{Main_Theorem} \label{Main_Theorem}
The only three-dimensional simply connected noncompact Type I $\kappa$-solution is the shrinking cylinder.
\end{Main_Theorem}

\bigskip

It is worth mentioning that Ni \cite{ni2009closed} has proved that a closed Type I $\kappa$-solution with positive curvature operator of every dimension is a shrinking sphere or one of its quotients. On the other hand, Theorem 2.4 in Ding \cite{ding2009remark} implies that the only simply connected noncompact $\kappa$-solution that forms a \textit{forward} singularity of Type I is the shrinking cylinder, Cao, Chow, and Zhang \cite{cao2016three} gave an alternative proof with an additional assumption of backward Type I. Furthermore, the author would like to draw the readers' attention to Hallgren \cite{hallgren2018nonexistence}, who also classified three-dimensional Type I $\kappa$-solution to the Ricci flow independently, through a more direct approach.
\\

We recall the notion of an $\varepsilon$-neck.

\newtheorem{Definition_2}[Definition_1]{Definition}
\begin{Definition_2}\label{Definition_2}
A space-time point $(x_0,t_0)$ in a Ricci flow $(M,g(t))$ is called the center of an $\varepsilon$-neck, where $\varepsilon>0$, if the Ricci flow $g(t)$ on the space-time neighbourhood $B_{g(t_0)}(x_0,\varepsilon^{-1}R(x_0,t_0)^{-\frac{1}{2}})\times[t_0-R(x_0,t_0)^{-1},t_0]$ is, after parabolic rescaling by the factor $R(x_0,t_0)$, $\varepsilon$-close in the $C^{\lfloor\frac{1}{\varepsilon}\rfloor}$-topology to the corresponding part of a standard shrinking cylinder, or in other words, if there exist diffeomorphisms $\phi_t:\mathbb{S}^2\times(-\varepsilon^{-1},\varepsilon^{-1})\rightarrow\\ B_{g(t_0)}(x_0,\varepsilon^{-1}R(x_0,t_0)^{-\frac{1}{2}})$, such that
\begin{eqnarray*}
&\phi_t^{-1}(x_0)\in\mathbb{S}^2\times\{0\},
\\
&\Big|R(x_0,t_0)\phi_t^*g(t_0+tR(x_0,t_0)^{-1})-g_{cyl}(t)\Big|_{C^{\lfloor\frac{1}{\varepsilon}\rfloor}(\mathbb{S}^2\times(-\varepsilon^{-1},\varepsilon^{-1}))} <\varepsilon,
\end{eqnarray*}
for any $t\in[-1,0]$. Here the notation $B_{g(t_0)}(x_0,r)$ stands for the geodesic ball centered at $x_0$, with radius $r$, and with respect to the metric $g(t_0)$, and $g_{cyl}(t)$ represents the standard shrinking metric on $\mathbb{S}^2\times\mathbb{R}$ with $R(g_{cyl}(0))\equiv 1$.
\end{Definition_2}

We remark here that in the above definition, after parabolic scaling, the space-time neighbourhood $\displaystyle B_{g(t_0)}(x_0,\varepsilon^{-1}R(x_0,t_0)^{-\frac{1}{2}})\times[t_0-R(x_0,t_0)^{-1},t_0]$ has time expansion $1$, and the scalar curvature at $(x_0,t_0)$ is normalized to be $1$. This definition is called the strong $\varepsilon$-neck by Perelman \cite{perelman2002entropy}, whereas we keep consistency with the definition in Kleiner-Lott \cite{kleiner2014singular} and call it an $\varepsilon$-neck.
\\

The following neck stability theorem by Kleiner and Lott is of fundamental importance to our proof. Please refer to Theorem 6.1 in \cite{kleiner2014singular}.

\newtheorem{Neck_Stability}[Definition_1]{Theorem}
\begin{Neck_Stability}\label{Neck_Stability}
For any $\kappa>0$, there exists a constant $\delta=\delta(\kappa)>0$, such that for all $\delta_0$, $\delta_1\leq\delta$, there is a $T=T(\delta_0,\delta_1,\kappa)\in(-\infty,0)$, with the following property. Let $(M^3,g(t))_{t\in(-\infty,0]}$ be a noncompact three-dimensional $\kappa$-solution to the Ricci flow that is not the $\mathbb{Z}_2$-quotient of the shrinking cylinder. Let $(x_0,0)\in M\times\{0\}$ be such that $R(x_0,0)=1$. If $(x_0,0)$ is the center of an $\delta_0$-neck, then for all $t\leq T$, $(x_0,t)$ is the center of a $\delta_1$-neck.
\end{Neck_Stability}

\bigskip

For the remaining of this paper, we fixed a small positive constant $\displaystyle\varepsilon <\min\left\{\frac{1}{100},\delta(\kappa),\varepsilon_0(\kappa)\right\}$, where $\delta(\kappa)$ is defined in Theorem \ref{Neck_Stability}, and $\varepsilon_0$ is the constant given in Corollary 48.1 of Kleiner and Lott \cite{kleiner2008notes}. With such $\varepsilon$ we are guaranteed that the $\varepsilon$-canonical neighbourhood property holds for all $\kappa$-solutions of dimension three. We will use this $\varepsilon$ as the small positive constant in the definition of the $\varepsilon$-neck.
\\

The following lemma is inspired by Ding \cite{ding2009remark} and Ni \cite{ni2009closed}.

\newtheorem{Lemma}[Definition_1]{Lemma}
\begin{Lemma}\label{Lemma}
Let $(M^3,g(t))_{t\in(-\infty,0]}$ be a three-dimensional noncompact Type I $\kappa$-solution with strictly positive sectional curvature on every time slice. Let $p_0$ be an arbitrary fixed point on $M$. Then for every instance $t\in(-\infty,0]$, there exists a point $p(t)\in M$ such that $(p(t),t)$ is \textbf{\emph{not}} the center of a $\varepsilon$-neck. Moreover, $\displaystyle dist_{g(0)}(p_0,p(t))\rightarrow\infty$ as $t\rightarrow-\infty$.
\end{Lemma}

\begin{proof}
First of all, such $p(t)$ must exists for every $t\in(-\infty,0]$. We know from the Gromoll-Meyer theorem that $M$ is diffeomorphic to $\mathbb{R}^3$. By Corollary 48.1 in Kleiner and Lott \cite{kleiner2014singular}, such ancient solution must fall into category $B$, on which there is always a cap (the so-called $M_\varepsilon$). In particular, since $M$ is diffeomorphic to $\mathbb{R}^3$, the cap is topologically a disk instead of $\mathbb{R}P^3\setminus B^3$.
\\

Assume by contradiction that there exists $\{t_i\}_{i=1}^\infty\subset(-\infty,0)$, such that $t_i\searrow-\infty$ but $\displaystyle dist_{g(0)}(p_0,p(t_i))\leq C_1$, where $C_1$ is a constant. We prove the following claim.

\newtheorem*{Claim}{Claim}
\begin{Claim}
There exists a constant $C_2<\infty$, such that
\begin{eqnarray}
dist_{g(t_i)}(p_0,p(t_i))\leq C_2\sqrt{|t_i|} +C_1,
\end{eqnarray}
for every $i$.
\end{Claim}
\begin{proof}[Proof of the Claim]
We recall Perelman's distance distortion estimate \cite{perelman2002entropy}. Suppose on $t_0$-slice of a Ricci flow, around two points $x_0$, $x_1$ that are not too close to each other, the Ricci curvature tensor is bounded from above, that is, if for some $r>0$, $dist_{g(t_0)}(x_0,x_1)\geq 2r$ and $Ric\leq (n-1)K$ on $B_{g(t_0)}(x_0,r)\bigcup B_{g(t_0)}(x_1,r)$, then we have
\begin{eqnarray} \label{Distortion}
\frac{d}{dt}dist_{g(t)}(x_0,x_1)\geq -2(n-1)\left(\frac{2}{3}Kr+r^{-1}\right)
\end{eqnarray}
at time $t=t_0$. Applying the curvature bound (\ref{Type I}) and $r=|t|^{\frac{1}{2}}$ to (\ref{Distortion}), we have
\begin{eqnarray} \label{distortion}
\frac{d}{dt}dist_{g(t)}(p_0,p(\tau_i))\geq -4 \left(C+1\right)|t|^{-\frac{1}{2}},
\end{eqnarray}
for every $i$, whenever $\displaystyle dist_{g(t)}(p_0,p(t_i))>2|t|^{\frac{1}{2}}$. Integrating (\ref{distortion}) from $0$ to $t_i\in(-\infty,0)$ completes the proof of the claim.
\end{proof}

Now we recall Perelman's reduced distance function $l_{(p_0,0)}(p,t)$ centered at $(p_0,0)$ and evaluated at $(p,t)$; see \cite{perelman2002entropy}. By the estimate of Naber (see Proposition 2.2 in \cite{naber2010noncompact}), we have that $l_{(p_0,0)}(p(t_i),t_i)<C_3$, where $C_3<\infty$ is a constant. It follows from Perelman \cite{perelman2002entropy} that there exists a subsequence of $\displaystyle \{(M,|t_i|^{-1}g(|t_i|t),(p(t_i),-1))_{t\in[-2,-1]}\}_{i=1}^\infty$ that converges in the pointed smooth Cheeger-Gromov sense to the canonical form of a nonflat shrinking gradient Ricci soliton; see Morgan and Tian \cite{morgan2007ricci} and Naber \cite{naber2010noncompact} for details. Notice that the time interval of these scaled flows are taken as $[-1,-\frac{1}{2}]$ in Perelman's argument, whereas we take the interval to be $[-2,-1]$, so as to keep consistency with the definition of the $\varepsilon$-neck. This is valid because $\displaystyle\sup_{t\in[2t_i,t_i]}l_{(p_0,0)}(p(t_i),t)$ is bounded uniformly. One may easy verify this bound by using Perelman's differential inequalities for the reduced distance. The only nonflat three-dimensional shrinking gradient Ricci solitons are the shrinking sphere, the shrinking cylinder, and their quotients; see Perelman \cite{perelman2003ricci}. The limit shrinking gradient Ricci soliton cannot be flat, since otherwise Perelman's reduced volume is equal to $1$ for all time and the Ricci flow is flat; see \cite{yokota2009perelman}. The shrinking cylinder is the only one that can arise as the limit of a sequence of Ricci flows that are diffeomorphic to $\mathbb{R}^3$. However, this yields a contradiction, as we have assumed that $(p(t_i),t_i)$ is not the center of an $\varepsilon$-neck.

\end{proof}

\bigskip

We are now ready to present the proof of our main theorem.

\begin{proof}[Proof of Theorem \ref{Main_Theorem}]
If $g(t)$ has zero sectional curvature somewhere in space-time, by Hamilton's strong maximum principle \cite{hamilton1986four}, $g(t)$ also has zero sectional curvature everywhere in space at more ancient times, and hence splits locally. Since we assume $M$ to be simply connected, it must be the shrinking cylinder. Therefore henceforth we assume that $g(t)$ has strictly positive curvature on every time slice.
\\

We fixed an arbitrary time sequence $\{t_i\}_{i=1}^\infty\subset(-\infty,0)$ such that $t_i\searrow-\infty$. For every $i$, let $p_i\in M$ be such that $(p_i,t_i)$ is \textbf{\textit{not}} the center of an $\varepsilon$-neck. By Lemma \ref{Lemma}, we have that $\displaystyle dist_{g(0)}(p_i,p_0)\rightarrow\infty$. Since by Perelman \cite{perelman2002entropy} that every three-dimensional noncompact $\kappa$-solution splits as a shrinking cylinder at spacial infinity, we can extract from $\{(M,R(p_i,0)g(t R(p_i,0)^{-1}),(p_i,0))_{t\in(-\infty,0]}\}_{i=1}^\infty$ a (not relabelled) subsequence that converges in the smooth Cheeger-Gromov sense to the shrinking round cylinder. For the sake of simplicity we denote $g_i(t):=R(p_i,0)g(t R(p_i,0)^{-1})$. It follows that for ever $i$ large, $(p_i,0)$ is the center of an $\varepsilon$-neck. The following claim is an easy consequence of Theorem \ref{Neck_Stability}.
\\

\newtheorem{Claim_1*}{Claim}
\begin{Claim}
\begin{eqnarray} \label{tau_i}
\bar{t}_i:=t_iR(p_i,0)\geq T,
\end{eqnarray}
for all large $i$.  Where $T:=T(\varepsilon,\varepsilon,\kappa)\in(-\infty,0)$ is defined in Theorem \ref{Neck_Stability}.
\end{Claim}

\begin{proof}[Proof of the claim]
Suppose the claim is not true, by passing to a subsequence, we can assume $\bar{t}_i=t_iR(p_i,0)< T$ for all $i$. We consider the scaled Ricci flows $g_i(t)$, and  apply Theorem \ref{Neck_Stability} to elements in $\{(M,g_i(t),(p_i,0))_{t\in(-\infty,0]}\}_{i=1}^\infty$. First of all we have that $R_i(p_i,0)=1$ because of the scaling factors that we chose. Moreover, forasmuch as $\{(M,g_i(t),(p_i,0))_{t\in(-\infty,0]}\}_{i=1}^\infty$ converges to the shrinking cylinder, we have that $(p_i,0)$ is the center of an $\varepsilon$-neck when $i$ is large. It follows that $(p_i,\bar{t}_i)$ is the center of an $\varepsilon$-neck when $i$ is large. However, $g_i(\bar{t}_i)=R(p_i,0)g(\bar{t}_i R(p_i,0)^{-1})=R(p_i,0)g(t_i)$. By our assumption, on the original Ricci flow $g(t)$ the space-time point $(p_i,t_i)$ is not the center of an $\varepsilon$-neck; this is a contradiction. Notice here that the $\varepsilon$-necklike property is scaling invariant.
\end{proof}
\bigskip

We continue the proof of the theorem. In the following argument we consider the scaled Ricci flows $g_i(t)$, notice that by our assumption for every $i$ the space-time point $(p_i,\bar{t}_i)$ is not the center of an $\varepsilon$-neck, where $\bar{t}_i$ is defined as (\ref{tau_i}). Since the limit of the sequence $\{(M,g_i(t),(p_i,0))_{t\in(-\infty,0]}\}_{i=0}^\infty$ is exactly a shrinking round cylinder, we have that for every large $A\in[4|T|,\infty)$, $(B_{g_i(0)}(p_i,A),g_i(t))_{\tau\in[T-A,0]}$ is as close as we like to the correspondent piece of the shrinking cylinder when $i$ is large enough. In particular, $(p_i,\bar{t}_i)$ is the center of an $\varepsilon$-neck since $\bar{t}_i\in[T,0]$ according to the claim; this is a contradiction. Here we have again taken into account the scaling invariance of the $\varepsilon$-necklike property.

\end{proof}

\bibliographystyle{plain}
\bibliography{citation}

\noindent Department of Mathematics, University of California, San Diego, CA, 92093
\\ E-mail address: \verb"yoz020@ucsd.edu"

\end{document}